\documentclass[a4paper,reqno,11pt,oneside]{amsart}
\usepackage{amsmath,geometry}
\usepackage{graphicx}
\usepackage{mathrsfs}
\usepackage{amssymb,amsmath, amsfonts, amsthm}
\usepackage{latexsym}
\usepackage{color} 
\usepackage[dvips]{epsfig}
\usepackage[colorlinks]{hyperref}
\usepackage{tikz}
\usepackage{pgfplots}

\usepackage{bm}
{ \theoremstyle{plain}
\newtheorem{theorem}{Theorem}[section]
\newtheorem{proposition}[theorem]{Proposition}
\newtheorem{lemma}[theorem]{Lemma}

  \theoremstyle{remark}
\newtheorem{remark}[theorem]{Remark}
  \theoremstyle{definition}

}

\renewcommand{\(}{\left(}
\renewcommand{\)}{\right)}

\begin{document}
\title[Segregated solutions for some non-linear critical Schr\"odinger systems]{Segregated solutions for some  non-linear Schr\"odinger systems with critical growth}

\author{Haixia Chen}
\address[H.Chen]{
 School of Mathematics and Statistics, Central China Normal University, Wuhan 430079, P. R. China.
}
\email{hxchen@mails.ccnu.edu.cn}

\author{Angela Pistoia}
\address[A.Pistoia]{Dipartimento di Scienze di Base e Applicate per l'Ingegneria, Sapienza Universit\`a di Roma, Via Scarpa 16, 00161 Roma, Italy}
\email{angela.pistoia@uniroma1.it}

\author{Giusi Vaira}
\address[G.Vaira]
{Dipartimento di Matematica, Universit\`a degli studi di Bari ``Aldo Moro'', via Edoardo Orabona 4,70125 Bari, Italy}
\email{giusi.vaira@uniba.it}

\begin{abstract}
We find infinitely many  positive non-radial solutions for a system of Schr\"odinger equations with critical growth  in a  fully attractive or repulsive regime   in presence of an external radial trapping potential. 
\end{abstract}

\date\today
 \thanks{H. Chen is partially supported by the NSFC grants (No.12071169) and the China Scholarship Council (No. 202006770017). The second and the third authors are partially supported by INDAM-GNAMPA funds. A. Pistoia is also partially supported by Fondi di Ateneo ``Sapienza" Universit\`a di Roma (Italy). G. Vaira is also partially supported by PRIN 2017JPCAPN003 ``Qualitative and quantitative aspects of nonlinear PDEs".\\
Data sharing not applicable to this article as no datasets were generated or analysed during the current study.}

\maketitle

    \section{Introduction}

The well-known Gross-Pitaevskii system  
\begin{equation}\label{GP}
- \iota \partial_t  \phi_i = \Delta  \phi_i -V_i(x) \phi_i+ \mu_i| \phi_i|^2 \phi_i+  \sum\limits_{j=1\atop j\not=i}^m\beta_{ij}
| \phi_j|^{2}  \phi_i ,\ i=1,\dots,m\end{equation}
 has been proposed as a mathematical model for multispecies
Bose-Einstein condensation in $m$ different  states. We refer to \cite{11,12,15,23} for a detailed physical motivation.
Here the complex valued functions $ \phi_i$'s are the wave functions of
the $i-$th condensate, $|\phi_i|$ is the amplitude of the $i$-th density, $ \mu_i$    describes  the interaction between particles of the same 
component
and $\beta_{ij}$, $i\not=j,$  describes   the interaction between particles   of different components, which can be {\em attractive} if $\beta_{ij}>0$ or {\em repulsive} if
  $\beta_{ij}<0$. 
To obtain solitary wave solutions of the Gross-Pitaevskii system \eqref{GP} one sets  $ \phi_i(t,x) = e^{- \iota \lambda_i t} u_i(x)$ and the real functions $u_i$'s solve the system
\begin{equation}\label{S}  -\Delta u_i +\lambda_i u_i +V_i(x) u_i= \mu_i u_i^3+u_i   \sum \limits_{ j=1\atop j\not=i }^m \beta_{ij} u_j^2  \ \hbox{in}\ \mathbb R^n,\ i =1,\dots,m\end{equation}
 where $\mu_i >0$, $\lambda_i>0$, $ \beta_{ij} =\beta_{ji}\in\mathbb R$, $V_i\in C^0(\mathbb R^n)$ and $n\ge2.$
\\

  We are interested in the so-called {\em vector} solutions, i.e. $u_i\not\equiv0$ for any $i=1,\dots,m.$ 
 Indeed, if one component $u_i$ identically vanishes, the system \eqref{S} is
reduced to a system with $m-1$ components.
\\

In the last decades,   the 
nonlinear Schrödinger system \eqref{S} has been widely studied.   Most of the work has been done when the spacial dimension is $n=2$ or $n=3$ (in this case the cubic non-linearity has sub-critical growth) and in absence of potentials (i.e. all the $V_i$'s are zero).   We refer the reader to a couple of   recent papers   \cite{bmw,WW} where the authors provide an exhaustive list of  references.
On the other hand the non-autonomous case is much less studied. It  has been treated firstly by Peng and Wang \cite{PW} and more recently by Pistoia and Vaira \cite{PV}.  In particular, they considered radial trapping potentials $V_i$'s and built (via a careful Ljapunov-Schmidt procedure)
unbounded sequences of non-radial positive vector
solutions   in a fully repulsive (i.e. all the $\beta_{ij}$'s are negative) or attractive (i.e. all the $\beta_{ij}$'s are negative)  regime.
\\

In dimension $n=4$ the cubic growth is   critical   and the existence of solutions to  \eqref{S} is a much more difficult issue. The autonomous case has been    studied by Clapp and Pistoia \cite{CP} and Clapp and Szulkin \cite{CS}, who  proved the existence of vector non-radial solutions to \eqref{S} in a fully repulsive regime using an interesting variational approach.
Recently, Chen, Medina and Pistoia \cite{cmp} built (via a sophisticated Ljapunov-Schmidt procedure) a new type of non-radial solutions in a weak  repulsive regime  (i.e. some $\beta_{ij}$'s are  equal, negative and small). In the present paper, we will focus on 
 the existence of solutions in the non-autonomous case which is by far quite unexplored.\\
 
 We assume that all the coupling parameters $\beta_{ij}$'s are equal to a real number $\beta$ and all the trapping potentials $V_i$'s coincide with a positive and radially symmetric function  $V\in L^2({\mathbb R}^4)\cap C^2({\mathbb R}^4)$. Therefore the system \eqref{S} reduces to the system
 \begin{equation}\label{beg}
-\Delta u_i+V(x)u_i=u_i^3+\beta \sum_{j\neq i}u_i u_j^2\text{~in~}{\mathbb R}^4, i=1,...,m.
\end{equation}
We are going to  build  infinitely many non-radial solutions to \eqref{beg}, whose  building blocks are   the so-called bubbles
  \begin{align}\label{udx}U_{\delta, \xi}(x)=\frac{1}{\delta} U\(\frac{x-\xi}{\delta}\)\ \hbox{with}\ U(x)=\frac{\mathtt c}{1+|x|^2},\ {\mathtt c}=2\sqrt{2},\end{align}
 which are the positive solutions to the critical equation
 $$-\Delta u=u^{3} \text{~in~}{\mathbb R}^4.$$
We assume that 
\begin{equation}\label{r0}
r_0>0\ \hbox{is a  non-degenerate critical point of the function}\ r\to r^2V(r).
\end{equation}
Now, we can state our main result.

\begin{theorem}\label{thm1.1}
There exists $k_0>0$ such that for any even  integer $k\geq k_0$, there exists a solution $(u_1,...,u_m)\in [\mathcal D^{1,2}({\mathbb R}^4)]^m$ to problem \eqref{beg} of the form 
$$u_q\sim \sum\limits_{i=1}^k U_{\delta,   \xi^q_i}\ \hbox{for any}\ q=1,...,m.$$
where  the bubbles $U_{\delta,  \xi^q_i}$ are defined in \eqref{udx}.\\
All the bubbles have the same  blow-up rate  which satisfies
$$\delta=e^{-d_kk^2}\ \hbox{and}\ d_k\sim \frac{\mathfrak d}{ r_0^2V(r_0)}>0\ \hbox{as}\ k\to\infty$$
for some positive constant  $\mathfrak d$.
\\
The blow-up points  $\xi_i^q$  (see \eqref{xiq}) of  the     component  $u_q$ lie on a circle  $\Gamma_q$ whose distance  from the origin approaches $r_0$ as $k\to\infty.$
Moreover,  $\Gamma_p$ with $\Gamma_q$ is   an Hopf link if   $p\not=q.$  
\end{theorem}

 \begin{remark}
 The solutions we find are of segregated type in the sense that the components tend to segregated from each other leading to phase separation.
 This phenomena has been widely studied by Terracini and her collaborators in a series of papers 
(see for example \cite{CTV,29,37,38,39} and references therein) and naturally appears in the strongly repulsive case (i.e.  $\beta_{ij}\to -\infty$).
In our case the existence of this kind of solutions is not affected at all by the presence of the coupling parameter $\beta$: they do  exist in both fully repulsive (i.e. $\beta<0$) and attractive regime (i.e. $\beta>0$). This is due to the fact  that the regime of the system is entirely encrypted in  the non-local term  of \eqref{nlo} (see below) which does not appear  in the reduced problem (i.e. it is an higher order term in \eqref{ex1} and \eqref{pooo2}).  
It would be extremely interesting to exhibit examples of potential $V$ (maybe changing-sign) and/or different configurations of bubbles for which the existence of solutions of \eqref{nlo} strongly depends on  the sign of the parameter $\beta.$
\\
\end{remark}
\begin{remark}Let us describe the strategy of the proof.
First, we look for solutions having a subtle symmetry which allows to reduce the system \eqref{beg} to a single non-local equation.
We follow an idea introduced by Chen, Medina and Pistoia in \cite{cmp}.
More precisely, if
\begin{equation}\label{sq} {\mathscr S}_q:=\(\begin{matrix}\cos \frac{(q-1)\pi}m &-\sin\frac{(q-1)\pi}m &0&0\\
\sin\frac{(q-1)\pi}m &\cos\frac{(q-1)\pi}m &0&0\\
0&0&\cos\frac{(q-1)\pi}m &\sin\frac{(q-1)\pi}m \\
0&0&-\sin\frac{(q-1)\pi}m &\cos\frac{(q-1)\pi}m  \\
 \end{matrix}\right)\ \hbox{for any}\ q=1,\dots,m.\end{equation}
and  $u$ is an  even function (i.e. $u(x)=u(-x)$) which solves the non-local equation
\begin{align}\label{nlo}
-\Delta u+V(x)u=u^3+\beta u\sum_{q=2}^mu^2({\mathscr S}_qx), \text{~in~}{\mathbb R}^4,
\end{align}
 then the functions defined via \eqref{sq}
   \begin{equation}\label{uq}u_q(x)=u({\mathscr S}_qx)\ \hbox{for any}\  q=1,...,m\end{equation}
solve the system \eqref{beg}.
 
 Next, we build a solution to \eqref{nlo} whose shape resembles $k$ copies of  bubbles whose peaks are located at the vertices of a regular polygon placed in a great circles $\Gamma $ with radius $r$. 
 As a result the peaks of  the function $u_q$ which solves  the system \eqref{beg} and is defined via the symmetry \eqref{uq},   lie on the great circle $\Gamma_q={\mathscr S}_q\(\Gamma\)$ and different components concentrate along linked great circles (i.e. $\Gamma_q$ with $\Gamma_p$ is an Hopf link if $q\not=p$).
 In order to carry out the construction of the solution to \eqref{nlo} we are inspired by Peng, Wang and Yan \cite{PWY}, who build positive solutions to  the Schrödinger equation
$$-\Delta u+V(x)u=u^{n+2\over n-2}, \text{~in~}{\mathbb R}^n\ (n\ge5)
$$
combining  the classical Ljapunov-Schmidt procedure together with a clever use of local 
 Pohozaev identities   to find the algebraic equations
which determine the location of the peaks.
Here we follow the same general
approach to study the non-local Schrödinger equation \eqref{nlo}. However a substantial difference arises when we  write the local Pohozaev identities, due to the presence of the
non-local term. To overcome the problem we use in a delicate and clever way the subtle symmetries  owned by the solutions we aim to build (see Proposition \ref{propexp}).\\ \end{remark}

 \begin{remark}
If we replace  ${\mathscr S}_q$ in \eqref{sq} by
$$ {\mathscr T}_q:=\(\begin{matrix}\cos \frac{(q-1)\pi}m &-\sin\frac{(q-1)\pi}m &0&0\\
\sin\frac{(q-1)\pi}m &\cos\frac{(q-1)\pi}m &0&0\\
0&0&1&0 \\
0&0&0 &1 \\
 \end{matrix}\right)\ \hbox{for any}\ q=1,\dots,m.$$
 a result similar to Theorem \ref{thm1.1} can be proved once we choose the peaks as
$ \xi_i^q={\mathscr T}_q^{-1}\xi_i$ where $ \xi_i:= {\rho} \( \cos\frac{2\pi(i-1)} k, \sin\frac{2\pi(i-1)} k, 0, 0\).$
In this case we can also relax the assumption on the potential $V$ only requiring that it is radially symmetric in the first two variables in the spirit of \cite{PWY}. \\
 \end{remark}

 \begin{remark}
It is also worthwhile to point out that we  introduce the good weighted spaces where   the reduction procedure in dimension $n=4$  can be carried out
(as far as we know only the case $n\ge5$ has been treated in the literature). 
This choice is a delicate issue, since  as it is usual we can not use the standard Sobolev spaces because of  the arbitrary large number of bubbles in the solution
and
the presence of the linear term in  the critical equation \eqref{nlo}  in 4D is not an innocent matter due
  the slow decay of the bubbles. 
   \\
   \end{remark}

 {\em Notation.} In what follows we agree that $f\lesssim g$ or $f=\mathcal O(g)$ means $|f|\le c |g| $  for some positive constant $c$ independent of $k$ and $f\sim g$  means $f= g+o(g)$.

\section{Proof of Theorem \ref{thm1.1}}\label{sec2}
\subsection{Rewriting the non-local equation via the finite dimentional reduction method}

We will find solutions of \eqref{nlo} in the space of symmetric functions
\begin{align*} X:=\{u\in \mathcal D^{1,2}({\mathbb R}^4): u \text{~satisfies~}  \eqref{sy22},\ \eqref{sy33}\ \hbox{and}\ \eqref{sy55}\},\end{align*}
i.e.
  \begin{align}
&u(x_1, x_2, x_3, x_4)=u(x_1, -x_2, x_3, -x_4),\label{sy22}\\
&\label{sy33}u(x_1, x_2,x_3,x_4)=u(x_3,x_4,x_1,x_2)\\
&\label{sy55}u(x)=u({\mathscr R}_ix)\ \hbox{for any}\ i=1,...,k,\end{align}
 where
  $k$  is an even integer and 
  $$
{{\mathscr R}}_i=\left(\begin{matrix}\cos\frac{2\pi(i-1)}{k}&\sin\frac{2\pi(i-1)}{k}&0&0\\
-\sin\frac{2\pi(i-1)}{k}&\cos\frac{2\pi(i-1)}{k}&0&0\\
0&0&\cos\frac{2\pi(i-1)}{k}&\sin\frac{2\pi(i-1)}{k}\\
0&0&-\sin\frac{2\pi(i-1)}{k}&\cos\frac{2\pi(i-1)}{k} \\
 \end{matrix}\).$$

In particular, we are going to build a solution to \eqref{nlo} as
\begin{align}\label{u11}u=\underbrace{\sum_{i=1}^k \chi U_{\delta, \xi_i}}_{:=W}+\phi\ \hbox{as}\ k\to+\infty,\end{align}
where the  bubbles $U_{\delta, \xi_i}$ are defined in  \eqref{udx} whose blow-up points are
\begin{align*}
 \xi_i:= \frac{{\rho}}{\sqrt2}\( \cos\frac{2\pi(i-1)} k, \sin\frac{2\pi(i-1)} k, \cos\frac{2\pi(i-1)} k, \sin\frac{2\pi(i-1)} k\),\ \hbox{with}\ |\rho-r_0|\le \vartheta \end{align*}
for some $\vartheta>0$ small and  blow-up rate satisfies $\delta=  e^{-dk^{2}}$ with $d\in [d_0, d_1]$ for some $d_1> d_0>0$. Here $r_0$ is given in \eqref{r0}.
Moreover, $\chi(x)=\chi(|x|)$  is a  radial   cut-off function whose support is close to the sphere $\{|x| =r_0\}$, namely 
  \begin{equation}\label{zet}\chi=1 \text{~in~}|r-r_0|\leq \sigma\ \hbox{and}\ \chi=0\text{~in~}|r-r_0|>2\sigma\ \hbox{for some $\sigma>0$ small.}\end{equation}
  Finally, as usual, $\phi$ is an higher order term.
  
  It is worthwhile to point out that the function $u_q$ given in \eqref{uq} 
  blows-up at the points
\begin{align}\label{xiq}\xi_i^q:={\mathscr S}_q^{-1}\xi_i=\frac{{\rho}}{\sqrt2} &\bigg(\cos\(-\frac{(q-1)\pi}m+ \frac{2\pi(i-1)} k\) ,\sin\(-\frac{(q-1)\pi}m+ \frac{2\pi(i-1)} k\),\nonumber\\
&\cos\(\frac{(q-1)\pi}m+ \frac{2\pi(i-1)} k\) ,\sin\(\frac{(q-1)\pi}m+ \frac{2\pi(i-1)} k\)\bigg). \end{align}
Plugging $u=W+\phi$ into  the non-local equation \eqref{nlo}, it can be rewritten as
\begin{align}\label{len}
\mathcal L(\phi)=\mathcal E+\mathcal N(\phi), \text{~in~}{\mathbb R}^4\end{align}
where $\mathcal L(\phi), \mathcal E, \mathcal N(\phi)$ are defined as
\begin{align}
&\mathcal L(\phi):=-\Delta\phi+V(x)\phi-3W^2\phi-\beta\phi \sum_{q=2}^mW^2({\mathscr S}_qx), \label{l}\\
&\mathcal E:=W^3+\Delta W-V(x)W+\beta W\sum_{q=2}^mW^2({\mathscr S}_qx), \label{e}\\
&\mathcal N(\phi):=\phi^3+3W\phi^2+\beta\phi\sum_{q=2}^m\phi^2({\mathscr S}_qx)+2\beta \phi\sum_{q=2}^mW({\mathscr S}_qx)\phi({\mathscr S}_qx)\nonumber\\
&~~~~~~~~~~~~~~~~~~~~~~~\ \ \ \ \ \ \ \ \ \ +\beta W \sum_{q=2}^m \phi^2({\mathscr S}_qx)+2\beta W \sum_{q=2}^mW({\mathscr S}_qx)\phi({\mathscr S}_qx).\label{n}
\end{align}
As it is usual, to solve \eqref{len} we will follow the classical steps of the Ljapunov-Schmidt procedure:
\begin{itemize}
\item[(i)] we show there  exists   $\phi\in X$ solution to the problem
\begin{equation}\label{exi}\left\{\begin{aligned}
&\mathcal L(\phi)=\mathcal E+\mathcal N(\phi)+\sum_{l=0}^1{\mathfrak c}_l(\delta, \rho)\sum_{i=1}^k(\chi U_{\delta, \xi_i})^2Z_{\delta,\xi_i}^l \\
& \int_{\mathbb R^4}\sum_{i=1}^k(\chi U_{\delta, \xi_i})^2Z_{\delta,\xi_i}^l\phi \ dx=0, l=0, 1\end{aligned}\right.\end{equation}
where $$Z^0_{\delta, \xi_i}=\frac{\partial (\chi U_{\delta, \xi_i})}{\partial \delta}\ \hbox{and}\ Z^1_{\delta, \xi_i}=\frac{\partial (\chi U_{\delta, \xi_i})}{\partial \rho} 
$$
\item[(ii)] we find $\delta>0$ and $\rho$ close to $r_0$ such that   ${\mathfrak c}_0(\delta, \rho)={\mathfrak c}_1(\delta, \rho)=0$.
\end{itemize} 
\subsection{Preliminaries}
Let us introduce the norms
\begin{align*} \|\phi\|_*=\sup\limits_{x\in\mathbb R^4} \bigg( \sum_{q=1}^m\sum\limits_{i=1}^k\frac{1}{ \delta+|x-{\mathscr S}_q^{-1}\xi_i|} \bigg)^{-1}|\phi(x)|,\end{align*}
\begin{align*}\|h\|_{**}=\sup\limits_{x\in\mathbb R^4} \bigg( \sum_{q=1}^m\sum\limits_{i=1}^k\frac{1}{(\delta+|x-{\mathscr S}_q^{-1}\xi_i|)^{3}} \bigg)^{-1}|h(x)|.\end{align*}

 The estimates of the norm of varies quantities rely on a couple of  important results whose proofs can be found in Appendix B of \cite{WY}.
\begin{lemma}\label{app1} For any $0<\alpha\leq\min\{\alpha_1, \alpha_2\}, i\neq j,$ it holds
$$\begin{aligned}&\frac{1}{(1+|x-x_i|)^{\alpha_1}}\frac{1}{(1+|x-x_j|)^{\alpha_2}}\\
&~~~~~\lesssim \frac{1}{|x_i-x_j|^\alpha}\(\frac{1}{(1+|x-x_i|)^{\alpha_1+\alpha_2-\alpha}}+\frac{1}{(1+|x-x_j|)^{\alpha_1+\alpha_2-\alpha}}\)\end{aligned}$$
\end{lemma}
\begin{lemma}\label{app2}For any constant $0<\alpha<2$, there is a constant $C>0$, such that
$$
\int_{\mathbb{R}^{4}} \frac{1}{|z-x|^{2}} \frac{1}{(1+|z|)^{2+\alpha}} d z \leqslant \frac{C}{(1+|x|)^{\alpha}}.
$$
\end{lemma}

Finally, we also state the following result which will be  used  throughout the paper.
\begin{lemma}\label{lem11}
There exists a positive constant $c, \zeta, C_\alpha$ such that
\begin{equation}\label{dis}|\xi_i^p-\xi^q_j|\left\{\begin{aligned} &\geq c>0 \ \ \text{~if~}p\neq q,\\&\geq \frac{\zeta}{k} \ \ \text{~if~}p=q \hbox{~and~}i\neq j.\end{aligned}\right.\end{equation}
and \begin{equation}\label{tau}
 \sum_{(q, i)\neq(p, j)}\frac{1}{|\xi_j^p-\xi_i^q|^{\alpha}}\left\{\begin{aligned}& \sim C_\alpha\frac{k^\alpha}{\rho^\alpha}\ \hbox{if}\ \alpha>1,\\
&\lesssim k\ln k\ \hbox{if}\ \alpha=1,\\
&\lesssim  k^{1+\alpha}\ \hbox{if}\ \alpha< 1.\end{aligned}\right.
\end{equation}

\end{lemma}
\begin{proof}
It is enough to remark that $$ |\xi_i^p-\xi^q_j|^2=2{\rho}^2\(1-\cos\frac{(p-q)\pi}m\cos\frac{(i-j)2\pi}k\)$$
and $$\begin{aligned} \sum_{(q, i)\neq(p, j)}\frac{1}{|\xi_j^p-\xi_i^q|^{\alpha}}&= \sum_q\sum_{i\not=j}\frac{1}{|\xi_j^q-\xi_i^q|^{\alpha}}+
\sum_{p\not=q}\sum_{i,j}\frac{1}{|\xi_j^p-\xi_i^q|^{\alpha}}\\
&= \sum_q\sum_{i\not=j}\frac{1}{|\xi_j^q-\xi_i^q|^{\alpha}}+\mathcal O(k)\\
&
\sim A_\alpha \frac{k^{\alpha}}{\rho^\alpha}\sum_{i=2}^k\frac{1}{i^\alpha}\end{aligned} $$
for some constant $A_\alpha$ depending on $\alpha.$

 \end{proof}
 
In the following we set $\eta^q_j=\frac{\xi^q_j}{\delta}$ and $\eta_i=\frac{\xi_i}{\delta}$ if $ i,j=1,...k$ and $q=1,...,m.$

\subsection{The size of the error term
 $\mathcal E$} 
\begin{proposition}\label{properr}Let $\mathcal E$ be defined as in \eqref{e}. Then
$$\|\mathcal E\|_{**} \lesssim  \delta .$$
\end{proposition}
\begin{proof}
We have
$$\mathcal E=\underbrace{W^3+\Delta W}_{=\mathcal E_1}-\underbrace{V(x)W}_{=\mathcal E_2}+\underbrace{\beta W\sum\limits_{q=2}^mW^2({\mathscr S}_q(x))}_{=\mathcal E_3}$$
\\
First  
$$ \mathcal E_1:=W^3+\Delta W=
\underbrace{\(\sum\limits_{i=1}^k\chi U_{\delta, \xi_i}\)^3-\chi\sum_{i=1}^k U_{\delta, \xi}^3}_{=I_1}+\underbrace{\Delta \chi \sum_{i=1}^kU_{\delta, \xi_i}+2\nabla \chi\sum_{i=1}^k\nabla U_{\delta, \xi_i}}_{=I_2}.$$
Now,
$$\begin{aligned}|I_1|\lesssim  \(\sum_{j\neq i}U_{\delta, \xi_i}^2U_{\delta, \xi_j}+(\chi^3-\chi)\sum_{i=1}^k U_{\delta, \xi_i}^3\),\end{aligned}$$
with $$(\chi^3-\chi)\sum_{i=1}^k U_{\delta, \xi_i}^3 \lesssim \delta^3\sum_{i=1}^k\frac{1}{(\delta+|x-\xi_i|)^3}$$
and for any $x\in\mathbb R^4$, denote $\omega=\frac{x}{\delta}$,   by Lemma \ref{app1} and \eqref{tau}
\begin{align*}
U_{\delta, \xi_1}^2\sum_{j\neq 1}U_{\delta, \xi_j}=& \frac{1}{\delta^3}\frac{1}{(1+|\omega-\eta_1|^2)^2}\sum_{j\neq1}\frac{1}{1+|\omega-\eta_j|^2}\\
\lesssim &\frac{1}{\delta^3}\sum_{j\neq 1}\frac{1}{|\eta_1-\eta_j|^\alpha}\bigg(\frac{1}{(1+|\omega-\eta_1|)^{6-\alpha}}+\frac{1}{1+|\omega-\eta_j|^{6-\alpha}}\bigg)\\
\lesssim &\delta^\alpha k^{1+\alpha}\sum_{j=1}^k\frac{1}{(\delta+|x-\xi_j|)^3}\text{~for any~}1<\alpha\leq 2,
\end{align*}
and so using the rotation transform \begin{align*}\sum_{i=1}^k\sum_{j\neq i}U_{\delta, \xi_i}^2U_{\delta, \xi_j}\lesssim k\cdot\delta^\alpha k^{1+\alpha}\sum_{j=1}^k\frac{1}{(\delta+|x-\xi_j|)^3},\end{align*}
i.e.,\begin{align*}|I_1|\leq \sum_{i=1}^k\frac{1}{(\delta+|x-\xi_i|)^3}\delta^\alpha k^{2+\alpha}\text{~for any~}1<\alpha\leq2.\end{align*}
Moreover, $$\begin{aligned}|I_2|&\lesssim  \bigg[\delta\sum_{i=1}^k\frac{1}{(\delta+|x-\xi_i|)^3} \cdot (\delta+|x-\xi_i|)|\Delta\chi|\\
&\ \ \ +\delta\sum_{i=1}^k\frac{1}{(\delta+|x-\xi_i|)^3}\cdot \frac{|x-\xi_i|}{\delta+|x-\xi_i|}|\nabla\chi|\bigg]\\
&\lesssim \delta\sum_{i=1}^k\frac{1}{(\delta+|x-\xi_i|)^3}.\end{aligned}$$
Therefore $$\|\mathcal E_1\|_*\leq\|I_1\|_*+\|I_2\|_*\lesssim   \delta.$$

Next \begin{align*}
|\mathcal E_2|\lesssim   \sum_{i=1}^k\frac{\delta}{\delta^2+|x-\xi_i|^2}  |\chi|\lesssim  \delta\sum_{i=1}^k\frac{1}{(\delta+|x-\xi_i|)^3}.\end{align*}
Finally  by  \eqref{dis} we get
$$|\mathcal E_3|\lesssim  \sum_{q=1}^m\sum_{i=1}^k\frac{1}{(\delta+|x-\xi_i^q|)^3}\delta^\alpha k^{2} \text{~for any~}1<\alpha\leq 2.$$
Collecting all the previous estimates, it follows
$$\|\mathcal E\|_*\leq \|\mathcal E_1\|_*+\|\mathcal E_2\|_*+\|\mathcal E_3\|_*\lesssim  \delta.$$
\end{proof}
\subsection{Solving a linear problem}
Denote  by $\(\Lambda_i\)_{i\in\mathbb N}$ the sequences of eigenvalues of the problem
\begin{align*}-\Delta Z=\Lambda_i U^2Z\ \text{~in~}{\mathbb R}^4,\ Z\in\mathcal D^{1,2}({\mathbb R}^4).\end{align*}
It is well known that the first eigenvalue is $ \Lambda_1=1$ and the associated  eigenspace is generated by the function $U.$ The second eigenvalue is $\Lambda_2=3$ and the associated eigenspace is generated by the function 
$$Z_0(x)={1-|x|^2\over (1+|x|^2)^2},\ Z_1(x)={x_1\over (1+|x|^2)^2},\ \dots,\ Z_4(x)={x_4\over (1+|x|^2)^2}.$$

\begin{proposition}\label{proplin}
Let  $\mathcal L(\phi)$ be  defined as in \eqref{l}. Assume $\beta\not\in \(\Lambda_i\)_{i\in\mathbb N}$. There exist $k_0>0$ and a constant $C>0$ independent of $k$ such that for any even $k\geq k_0$, $\delta=e^{-dk^2}$ with $d\in [d_0, d_1]$ for some $d_1> d_0>0$, $\rho\in(r_0-\vartheta,r_0+\vartheta)$ with $\vartheta>0$ and for any $h\in L^{\frac43}({\mathbb R}^4)$ satisfying \eqref{sy22}-\eqref{sy55}, the linear problem
\begin{equation}\label{phi}\left\{\begin{aligned}&\mathcal L(\phi)=h+\sum_{l=0}^1{\mathfrak c}_l\sum_{i=1}^k(\chi U_{\delta, \xi_i})^2Z_{\delta,\xi_i}^l,\text{~in~} {\mathbb R}^4,\\&\int_{{\mathbb R}^4}\sum_{i=1}^k(\chi U_{\delta, \xi_i})^2Z_{\delta,\xi_i}^l\phi \ dx=0, l=0,1.\end{aligned}\right.\end{equation}
admits a unique solution $\phi\in X$ satisfying 
\begin{align}\label{con}
\|\phi\|_*\lesssim  \|h\|_{**}\text{~and~} |{\mathfrak c}_l|\lesssim  \delta\|h\|_{**}, l=0,1.
\end{align}
for some $\mathfrak c_0, \mathfrak c_1$.
\end{proposition}
\begin{proof}

\noindent \textbf{Step 1:} Assume first that \eqref{con}  holds. Define $$X_0:=\{\phi\in X, \|\phi\|_*<+\infty, \int_{{\mathbb R}^4}\sum_{i=1}^k(\chi U_{\delta, \xi_i})^2Z_{\delta,\xi_i}^l\phi \ dx=0, l=0,1\}.$$
The first equation in \eqref{phi} can be rewritten as
\begin{align*}\phi+\underbrace{(-\Delta)^{-1}(V(x)\phi-3W^2\phi -\beta\phi \sum_{q=2}^mW^2({\mathscr S}_qx))}_{:=\mathcal K\phi}=\underbrace{(-\Delta)^{-1}(h+\sum_{l=0}^1{\mathfrak c}_l\sum_{i=1}^k(\chi U_{\delta, \xi_i})^2Z_{\delta,\xi_i}^l)}_{:=f}\end{align*}
where $\mathcal K: X_0\to X_0$ is a compact operator for each fixed $k$ since $W^2, W^2({\mathscr S}_qx), V(x)\in L^2({\mathbb R}^4)$(see Lemma 2.3 in \cite{BS}) and there exist $\mathfrak c_0, \mathfrak c_1$ such that $f\in X_0$. Thus there is a unique $\phi\in X_0$ such that $\phi+\mathcal K\phi=f$ with $f\in X_0$ by Fredholm-alternative theorem.
\vspace{1mm}

Now, let us prove \eqref{con}. 
 Assume by contradiction that there exist a sequence $\phi_{n}$ satisfying \eqref{phi} with $h=h_{n}$, and $\|\phi_{n}\|_*=1, \|h_{n}\|_{**}\to 0$. For sake of  simplicity, we will drop the subscript $n$.

\vspace{2mm}

\noindent \textbf{Step 2:}  We claim that \begin{align}\label{cl1}{\mathfrak c}_l=\mathcal O(\delta^2k^{1+\tau}\|\phi\|_*)+\mathcal O(\delta\|h\|_{**}),\ l=0,1\ \hbox{for some small}\ \tau>0.\end{align}
Testing the first equation in \eqref{phi} by $Z^j_{\delta, \xi_1}, j=0,1$, we get
$$\begin{aligned}&\sum_{l=0}^1{\mathfrak c}_l\int_{{\mathbb R}^4}\sum_{i=1}^k(\chi U_{\delta, \xi_i})^2Z^l_{\delta, \xi_i}Z^j_{\delta, \xi_1}dx\\ &=\int_{{\mathbb R}^4}\(-\Delta\phi+V(x)\phi-3W^2\phi-\beta\phi \sum_{q=2}^mW^2({\mathscr S}_qx)-h\)Z^j_{\delta, \xi_1}dx.\end{aligned}$$

Concerning the L.H.S. it is immediate to check that there exist $B_l\neq 0, l=0,1$ such that
\begin{align*}
\int_{{\mathbb R}^4}\sum_{i=1}^k(\chi U_{\delta, \xi_i})^2Z^l_{\delta, \xi_i}Z^j_{\delta, \xi_1}dx=\left\{\begin{aligned} &B_l \delta^{-2}(1+o(1)), l=j,\\& o(\delta^{-2}), ~~~~~~~~~~~l\neq j.\end{aligned}\right.
\end{align*}

Now, let us look at the R.H.S.. First, by Lemma \ref{app1}, \eqref{tau}  and the fact that $|Z^j_{\delta, \xi_i}|\lesssim \frac{1}{\delta^2+|x-\xi_i|^2}, j=0,1$ 
\begin{align*}\int_{{\mathbb R}^4}h Z^j_{\delta, \xi_1}dx\lesssim&\|h\|_{**}\int_{{\mathbb R}^4}\sum_{q=1}^m\sum\limits_{i=1}^k\frac{1}{(\delta+|x-\xi_i^q|)^{3}}\frac{1}{\delta^2+|x-\xi_1|^2}dx\\
\lesssim&  \delta^{-1}\|h\|_{**}\int_{{\mathbb R}^4} \sum_{q=1}^m\sum\limits_{i=1}^k\frac{1}{(1+|\omega-\eta^q_i|)^{3}}\frac{1}{1+|\omega-\eta_1|^2}d\omega\\
\lesssim& \delta^{-1}\|h\|_{**}\int_{{\mathbb R}^4}\frac{1}{(1+|\omega-\eta_1|)^{5}}+\sum_{(q,i)\neq (1, 1)}\sum\limits_{i=2}^k\frac{1}{|\eta_1-\eta^q_i|^{\alpha}}\\
&~~~~~~~~~~~~~~~\cdot \(\frac{1}{(1+|\omega-\eta_1|)^{5-\alpha}}+\frac{1}{(1+|\omega-\eta^q_i|)^{5-\alpha}}\)d\omega\\
\lesssim& \delta^{-1}\|h\|_{**}\text{~for any~}0<\alpha<1,
\end{align*}
\begin{align}&
\int_{{\mathbb R}^4}V(x)\phi Z^j_{\delta, \xi_1}dx\nonumber\\ & \lesssim \|\phi\|_*\int_{|r-r_0|\leq 2\sigma}\sum_{q=1}^m\sum\limits_{i=1}^k\frac{1}{\delta+|x-\xi_i^q|}\frac{1}{\delta^2+|x-\xi_1|^2}dx\nonumber\\
&\lesssim  \|\phi\|_*\left[\int_{|r-r_0|\leq 2\sigma} \frac{1}{(\delta+|x-\xi_1|)^3}dx\right.\nonumber\\
&+\left.\sum_{(q,i)\neq (1, 1)}\frac{\delta^\alpha}{|\xi^q_i-\xi_1|^\alpha}\frac{\delta^4}{\delta^3}\int_{|r-r_0|\leq \frac{2\sigma}{\delta}}\(\frac{1}{(1+\frac{|x-\xi_i^q|}{\delta})^{3-\alpha}}+\frac{1}{(1+\frac{|x-\xi_1|}{\delta})^{3-\alpha}}\)dx\right]\nonumber\\
&\lesssim k^{1+\tau}\|\phi\|_*\text{~where $\tau>0$ is small for any $0<\alpha\leq 1$}. \label{vpz}
\end{align}
and
\begin{align*}
&\int_{{\mathbb R}^4}\phi\sum_{q=2}^mW^2({\mathscr S}_qx)Z^j_{\delta, \xi_1}dx\nonumber\\
&\lesssim \|\phi\|_*\int_{{\mathbb R}^4}\sum_{q=1}^m\sum_{i=1}^k\frac{1}{\delta+|x-\xi^q_i|}\sum_{p=2}^m\(\sum_{j=1}^k\frac{\delta}{\delta^2+|x-\xi^p_j|^2}\)^2 \frac{1}{\delta^2+|x-\xi_1|^2}dx\nonumber\\
&\lesssim \|\phi\|_*\int_{{\mathbb R}^4} \frac{1}{(\delta+|x-\xi_1|)^3}\sum_{p=2}^m\(\sum_{j=1}^k\frac{\delta}{\delta^2+|x-\xi^p_j|^2}\)^2dx \nonumber\\
 &\lesssim  \delta^{-1}k\|\phi\|_*\int_{{\mathbb R}^4} \sum_{q=2}^m\sum_{j=1}^k\frac{ 1}{|\eta_1-\eta^q_j|^\alpha} \(\frac{1}{(1+|\omega-\eta_1|)^{7-\alpha}}+\frac{1}{(1+|\omega-\eta^q_j|)^{7-\alpha}}\)d\omega  \nonumber\\
&\lesssim \delta^{\alpha-1}k^{1+\alpha}\|\phi\|_*
\text{~for any $1<\alpha<3$.} 
\end{align*} 
Next, let  us consider the term
$$\int_{{\mathbb R}^4}(-\Delta\phi-3W^2\phi)Z^j_{\delta, \xi_1}dx, j=0,1.$$
If $j=0$ (the case  $j=1$ is similar), arguing as in \eqref{vpz}  for $\tau>0$ small enough
$$\int_{{\mathbb R}^4}(-\Delta \chi)\frac{\partial U_{\delta, \xi_1}}{\partial\delta}\phi dx=\mathcal O(k^{1+\tau}\|\phi\|_*)$$
and then 
\begin{align}
&\int_{{\mathbb R}^4}(-\Delta\phi-3W^2\phi)Z^0_{\delta, \xi_1}dx\nonumber\\
&=\int_{{\mathbb R}^4}\phi\(\frac{\partial(-\Delta\chi  U_{\delta, \xi_1}-2\nabla \chi\nabla U_{\delta, \xi_1}+\chi U^3_{\delta, \xi_1})}{\partial \delta}-3W^2Z^1_{\delta, \xi_1}\)\nonumber\\
&=\int_{{\mathbb R}^4}\left[(3\chi U_{\delta, \xi_1}^2\frac{\partial U_{\delta, \xi_1}}{\partial \delta}-3\chi^3U_{\delta, \xi_1}^2\frac{\partial U_{\delta, \xi_1}}{\partial \delta})\phi\right]+\left[3(\chi^2U_{\delta, \xi_1}^2-W^2)\chi\frac{\partial U_{\delta, \xi_1}}{\partial \delta}\phi \right]dx\nonumber\\&\ \ \ \ \ \ \ \ +
\int_{{\mathbb R}^4}(-\Delta \chi)\frac{\partial U_{\delta, \xi_1}}{\partial\delta}\phi dx-\int_{{\mathbb R}^4}2\nabla \chi \nabla\frac{\partial U_{\delta, \xi_1}}{\partial\delta}\phi dx\nonumber\\
&=o(\|\phi\|_*)+\mathcal O(k^{1+\tau}\|\phi\|_*)+\mathcal O(k\|\phi\|_*)=\mathcal O(k^{1+\tau}\|\phi\|_*). \label{deph}
\end{align}
Finally \eqref{cl1} follows by all the above estimates.
\vspace{2mm}

\noindent \textbf{Step 3:} We claim that if $\tau>0$ is small enough 
\begin{align}\label{apri}|\phi(x)|\leq  (o(\|\phi\|_*)+\mathcal O(\|h\|_{**}))\sum\limits_{q=1}^m\sum\limits_{i=1}^k\frac{1}{\delta+|x-\xi_i^q|}+\mathcal O\bigg(\sum\limits_{q=1}^m\sum\limits_{i=1}^k\frac{\delta^\tau}{\delta+|x-\xi_i^q|^{1+\tau}}\bigg).\end{align}

Indeed, since $V$ is positive, by    \cite[Lemma C.1]{MWY}, 
\begin{align*}
|\phi(x)|&\lesssim \bigg|\int\limits_{{\mathbb R}^4}\frac{1}{|x-y|^2}\left[\(3W^2\phi+\beta\phi \sum_{q=2}^mW^2({\mathscr S}_qx)\) +h+\sum_{l=0}^1{\mathfrak c}_l\sum_{i=1}^k(\chi U_{\delta, \xi_i})^2Z_{\delta,\xi_i}^l \right]dy\bigg|\\
&:=II_1+II_2+II_3.
\end{align*}
Since for small $\tau>0$, for any $x\in\mathbb R^4$, using \eqref{tau} and Lemma \ref{app1}, \begin{align*}&\sum_{q=1}^m\bigg(\sum_{i=1}^k\frac{\delta}{\delta^2+|y-\xi^q_{i}|^2}\bigg)^2\sum_{p=1}^m\sum_{j=1}^k\frac{1}{\delta+|y-\xi^p_{j}|}\\
&\lesssim\sum_{q=1}^m\sum_{i=1}^k\frac{\delta^2}{(\delta+|y-\xi^q_i|)^5}+k\sum_{q=1}^m\sum_{i=1}^k\bigg(\frac{\delta}{\delta^2+|y-\xi^q_i|^2}\bigg)^2\sum_{(p,j)\neq(q,i)}\frac{1}{\delta+|y-\xi^p_j|}\\
&\lesssim\frac{1}{\delta^3}\bigg[\sum_{q=1}^m\sum_{i=1}^k\frac{1}{(1+|\omega-\eta^q_i|)^5}+k\sum_{(p,j)\neq(q,i)}\frac{1}{|\eta^q_i-\eta^p_j|^\alpha}\\
&\ \ \ \ \ \ \ \ \ \cdot\bigg(\frac{1}{(1+|\omega-\eta^q_i|)^{5-\alpha}}+\frac{1}{(1+|\omega-\eta^p_j|)^{5-\alpha}}\bigg)\bigg]\\
&\lesssim \frac{1}{\delta^3}\sum_{q=1}^m\sum_{i=1}^k\frac{1}{(1+|\omega-\eta^q_i|)^{3+\tau}}(1+ \delta^\alpha k^{3+\alpha})\text{~for some~}0<\alpha<1,
\end{align*}

 By  Lemma \ref{app2}  
\begin{align*}
II_1&\lesssim  \|\phi\|_*\bigg|\int_{{\mathbb R}^4}\frac{1}{|x-y|^2}\sum_{q=1}^m\bigg(\sum_{i=1}^k\frac{\delta}{\delta^2+|y-\xi^q_{i}|^2}\bigg)^2\sum_{p=1}^m\sum_{j=1}^k\frac{1}{\delta+|y-\xi^p_{j}|} dy \bigg|\nonumber\\
 &\lesssim \|\phi\|_*\sum_{q=1}^m\sum_{i=1}^k\frac{\delta^{\tau}}{(\delta+|x-\xi^q_{i}|)^{1+\tau}} \text{~for small ~}\tau>0 
\end{align*}
and
\begin{align*}
II_2\lesssim & \|h\|_*\bigg|\int_{{\mathbb R}^4}\frac{1}{|x-y|^2} \sum_{q=1}^m\sum_{i=1}^k\frac{1}{(\delta+|y-\xi^q_{i}|)^3} dy \bigg|\nonumber\\
\lesssim &\|h\|_*\sum_{q=1}^m\sum_{i=1}^k\frac{1}{\delta+|x-\xi^q_{i}| }. 
\end{align*}

Next, taking into account that $|Z^l_{\delta, \xi_i}|\lesssim \frac{1}{\delta^2+|x-\xi_i|^2}, l=0, 1$, we    get
\begin{align}
 II_3&\lesssim \sum_{l=0}^1|{\mathfrak c}_l|\bigg|\int_{{\mathbb R}^4}\frac{1}{|x-y|^2}\sum\limits_{i=1}^k\frac{\delta^{2}}{(\delta^2+|y-\xi_i|^2)^{3}}dy\bigg|\nonumber\\
& \lesssim \sum_{l=0}^1|{\mathfrak c}_l|\delta^{-2}   \bigg|\int_{{\mathbb R}^4}\frac{1}{|\omega-\frac{x}{\delta}|^2}\sum\limits_{i=1}^k\frac{1}{(1+|\omega-\eta_i|)^{6}}d\omega\bigg|\nonumber\\
& \lesssim \sum_{l=0}^1|{\mathfrak c}_l|\delta^{-2}   \bigg|\int_{{\mathbb R}^4}\frac{1}{|\omega-\frac{x}{\delta}|^2}\sum\limits_{i=1}^k\frac{1}{(1+|\omega-\eta_i|)^{3+\tau}}d\omega\bigg|\nonumber\\
& \lesssim \sum_{l=0}^1|{\mathfrak c}_l|\delta^{-2} \sum\limits_{i=1}^k\frac{1}{(1+|\frac{x}{\delta}-\eta_i|)^{1+\tau}}\nonumber\\
& \lesssim \sum_{l=0}^1|{\mathfrak c}_l|\sum\limits_{i=1}^k\frac{\delta^{-1+\tau}}{(\delta+|x-\xi_i|)^{1+\tau}}.\nonumber
 \end{align}
Finally, \eqref{apri} follows by all the above estimates.
\vspace{2mm}

\noindent \textbf{Step 4:}  We claim that there exist $c, R>0$ and $\xi^{q_0}_{1}$ for some $q_0$ such that
$$\|\delta\phi(\delta y+\xi^{q_0}_{1})\|_{L^{\infty}(B_R(0))}\geq c>0.$$

Indeed, from  \eqref{apri}, we have
\begin{align*}
\bigg(\sum_{q=1}^m\sum_{i=1}^k\frac{1}{\delta+|x-\xi_i^q|}\bigg)^{-1}|\phi(x)|\lesssim  \sum_{q=1}^m\sum_{i=1}^k\frac{\delta^\tau}{(\delta+|x-\xi_i^q|)^\tau}, x\in \mathbb R^4.
\end{align*}
Let $\zeta>0$ be a small constant. We have
\begin{align*}
\sup_{x\in \Big(\cup_{q=1}^m\cup_{i=1}^kB(\xi^q_i, \frac{\zeta}{k})\Big)^c}\bigg(\sum_{q=1}^m\sum_{i=1}^k\frac{1}{\delta+|x-\xi_i^q|}\bigg)^{-1}|\phi(x)| \lesssim \delta^\tau k^{1+\tau}=o(1).
\end{align*}
On the other side, since $\phi$ satisfies \eqref{sy55} by \eqref{dis}, 
\begin{align*}
&\sup_{x\in \cup_{q=1}^m\cup_{i=1}^kB(\xi^q_i, \frac{\zeta}{k})}\bigg(\sum_{q=1}^m\sum_{i=1}^k\frac{1}{\delta+|x-\xi_i^q|}\bigg)^{-1}|\phi(x)| \\
&=\sup_{x\in \cup_{q=1}^mB(\xi^q_{i_0}, \frac{\zeta}{k})}\bigg(\sum_{q=1}^m\sum_{i=1}^k\frac{1}{\delta+|x-\xi_i^q|}\bigg)^{-1}|\phi(x)|\\
&\lesssim \sup_{x\in \cup_{q=1}^mB(\xi^q_{i_0}, \frac{\zeta}{k})}\sum_{q=1}^m\sum_{i=1}^k\frac{\delta^\tau}{(\delta+|x-\xi_i^q|)^\tau}\\
&\lesssim\sup_{x\in \cup_{q=1}^mB(\xi^q_{i_0}, \frac{\zeta}{k})}\bigg(\sum_{q=1}^m \frac{\delta^\tau}{(\delta+|x-\xi_{i_0}^q|)^\tau}+\sum_{p=1}^m\sum_{i\neq i_0}\frac{\delta^\tau}{(\delta+|x-\xi_i^p|)^\tau}\bigg)\\
&\lesssim\sup_{x\in \cup_{q=1}^mB(\xi^q_{i_0}, \frac{\zeta}{k})}\sum_{q=1}^m \frac{\delta^\tau}{(\delta+|x-\xi_{i_0}^q|)^\tau}+\delta^\tau k^{1+\tau},
\end{align*}
where the last inequality we used $ \frac{1}{k}\lesssim  |\xi^p_i-\xi^q_{i_0}|-|x-\xi^q_{i_0}| \lesssim |x-\xi^p_i|$ for $x\in \cup_{q=1}^mB(\xi^q_{i_0}, \frac{\zeta}{k})$.
It gives from $\|\psi\|_{\star}=1$ 
\begin{align*}
1\lesssim \sup_{x\in \cup_{q=1}^mB(\xi^q_{i_0}, \frac{\zeta}{k})}\sum_{q=1}^m \frac{\delta^\tau}{(\delta+|x-\xi_{i_0}^q|)^\tau}\text{~for any~}i_0\in\{1,...,k\}
\end{align*}
which implies that there exist $x_*, \xi^{q_0}_{1}$ and $  R>0$ such that $ |x_*-\xi^{q_0}_{1}|\leq \delta R$.
Then, by \eqref{apri}  $$\begin{aligned}1\lesssim &\sup_{B_{\delta R}(\xi^{q_0}_i)}\bigg(\sum_{q=1}^m\sum_{i=1}^k\frac{1}{\delta+|x-\xi^q_i|}\bigg)^{-1}|\phi(x)| \\ \lesssim &\sup_{B_{\delta R}(\xi^{q_0}_{1})}\(\frac{1}{\delta+|x-\xi^{q_0}_1|}\)^{-1}|\phi(x)|\\
 \lesssim & (1+R) \|\delta \phi\|_{L^{\infty}(B_{\delta R}(\xi^{q_0}_{1}))},  \end{aligned}$$
which proves the claim.
\vspace{2mm}

\noindent \textbf{Step 5:} Let us prove that $\delta\phi(\delta y+\xi_{1}^{q_0})\to 0$ in $L^{\infty}(B_R(0))$ and a contradiction arises.

Indeed, let us consider the function $\tilde\phi(y)=\delta\phi(\delta y+\xi^{q_0}_{1})$. Since
$\|\phi\|_*=1$, by the standard regularity theories and Arzel\`a-Ascoli theorem, we can assume that, up to a subsequence,  $\tilde\phi \to \phi^*$ in $C^1(B_R(0))$ for any $R>0$ and $\phi^*$   solves either
\begin{align}\label{ca1}-\Delta \phi^*-3U^2\phi^*=0, \text{~in~}{\mathbb R}^4,\text{~if~}\xi^{q_0}_{1}=\xi_1\end{align}
or
\begin{align}\label{ca2}-\Delta \phi^*-\beta U^2\phi^*=0, \text{~in~}{\mathbb R}^4,\text{~if~}\xi^{q_0}_{1}\neq\xi_1.\end{align}
We claim that $\phi^*=0$ and a contradiction arises because of what we proved in  Step 4.
In  case \eqref{ca2}, since  $\beta\not\in\(\Lambda_i\)_{i\in\mathbb N}$, we immediately get that  $\phi^*=0$. In case \eqref{ca1}, 
first we observe that  $\phi$  and also  $\phi^*$ satisfy \eqref{sy22}, \eqref{sy33}
and so \begin{align*}
&\int_{{\mathbb R}^4}\frac{1}{(1+|y|^2)^2}\frac{y_2}{(1+|y|^2)^2}\phi^*dy =0=\int_{{\mathbb R}^4}\frac{1}{(1+|y|^2)^2}\frac{y_4}{(1+|y|^2)^2}\phi^*dy,\\
&\int_{{\mathbb R}^4}\frac{1}{(1+|y|^2)^2}\frac{y_1}{(1+|y|^2)^2}\phi^*dy =\int_{{\mathbb R}^4}\frac{1}{(1+|y|^2)^2}\frac{y_3}{(1+|y|^2)^2}\phi^*dy.
\end{align*} Moreover, since $\int_{{\mathbb R}^4}\phi \sum_{i=1}^k(\chi U_{\delta, \xi_i})^2Z^l_{\delta, \xi_i} dy=0, l=0, 1$, we have \begin{align*}
\int_{{\mathbb R}^4}\frac{1}{(1+|y|^2)^2}\frac{y_1+y_3}{(1+|y|^2)^2}\phi^*dy =0,\end{align*}
and so $\int_{\mathbb R^4}3U^2Z_l\phi^*dy=0, l=0,1,3$. Therefore $\phi^*=0$. 
\end{proof}
\subsection{Solving the non-linear problem}
We are going to solve the non-linear problem \eqref{exi} by using the standard fixed point theorem. 
\begin{proposition}\label{propfix}
Assume $\beta\not\in\(\Lambda_i\)_{i\in\mathbb N}$. There exist $k_0>0$ and a constant $C>0$ independent of $k$ such that for any even $k\geq k_0$, $\delta=e^{-dk^2}$ with $d\in [d_0, d_1]$ and $\rho\in(r_0-\vartheta,r_0+\vartheta)$ with $\vartheta>0$ the problem \eqref{exi} has a unique solution $\phi\in X$ satisfying
$$\|\phi\|_*\lesssim \delta,\ \ \  |\mathfrak c_l|\lesssim \delta^2, l=0,1.$$ 
\end{proposition}
\begin{proof}
The proof relies on a standard contraction mapping argument together with Proposition \ref{properr} and  Lemma \ref{nn} below.
\end{proof}

\begin{lemma}\label{nn}
Let  $\mathcal N(\phi)$ be  defined as in \eqref{n}. Then
$$ \|\mathcal N(\phi)\|_{**} \lesssim (\|\phi\|_*^2+k^2\|\phi\|_*^3+\delta^2k^4\|\phi\|_*).$$
\end{lemma}
\begin{proof}
First of all, since   $\|\phi({\mathscr S}_qx)\|_*=\|\phi\|_*$
and $$\sum_{l=1}^m\sum_{i=1}^k\frac{1}{\delta+|x-\xi^l_i|}=\sum_{p=1}^m\sum_{j=1}^k\frac{1}{\delta+|{\mathscr S}_qx-\xi^p_j|}, q=2,...,m,$$
we get
$$\begin{aligned}|\mathcal N(\phi)|\lesssim \bigg[&\|\phi\|_*^3\bigg(\sum\limits_{q=1}^m\sum\limits_{i=1}^k\frac{1}{\delta+|x-\xi^q_i|}\bigg)^3+\|\phi\|_*^2\sum_{q=1}^mW({\mathscr S}_qx)\bigg(\sum\limits_{q=1}^m\sum\limits_{i=1}^k\frac{1}{\delta+|x-\xi^q_i|}\bigg)^2\\
&+\|\phi\|_*W \sum_{q=2}^mW({\mathscr S}_qx)\bigg(\sum\limits_{q=1}^m\sum\limits_{i=1}^k\frac{1}{\delta+|x-{\mathscr S}_q^{-1}\xi^q_i|}\bigg)\bigg].\end{aligned}$$
By Lemma \ref{app1} and \eqref{tau} again, for any $x\in\mathbb R^4$, let $\omega=\frac{x}{\delta}$,
\begin{align*}&\sum_{q=1}^m\sum_{i=1}^k\frac{\delta}{\delta^2+|x-\xi_i^q|^2}\bigg(\sum_{p=1}^m\sum\limits_{j=1}^k\frac{1}{\delta+|x-\xi_j^p|}\bigg)^2\\
&\lesssim \sum_{q=1}^m\sum_{i=1}^k\frac{\delta}{\delta^2+|x-\xi_i^q|^2}\bigg[\frac{1}{(\delta+|x-\xi_i^q|)^2}+\(\sum_{(p,j)\neq(q,i)}\frac{1}{\delta+|x-\xi_j^p|}\)^2\bigg]\\
&\lesssim  \sum_{q=1}^m\sum_{i=1}^k\bigg[\frac{\delta}{(\delta+|x-\xi_i^q|)^4}+k\sum_{(p,j)\neq(q,i)}\frac{1}{|\xi_i^q-\xi_j^p|^\alpha}\\
&\ \ \ \ \ \cdot\bigg(\frac{\delta}{(\delta+|x-\xi_i^q|)^{4-\alpha}}+\frac{\delta}{(\delta+|x-\xi_j^p|)^{4-\alpha}}\bigg)\bigg]\\
&\lesssim  \sum_{q=1}^m\sum_{i=1}^k \frac{1}{(\delta+|x-\xi_i^q|)^3}(1+\delta^{\alpha}k^{2+\alpha})\text{~for any~}1<\alpha\leq 2, \end{align*}
and similarily from \eqref{dis}
\begin{align*}
 &W\sum_{q=2}^mW({\mathscr S}_qx)\bigg(\sum\limits_{q=1}^m\sum\limits_{i=1}^k\frac{1}{\delta+|x-{\mathscr S}_q^{-1}\xi^q_i|}\bigg)\\
&\lesssim \frac{k}{\delta^3} \bigg[\sum_{j=1}^k\frac{1}{(1+|\omega-\eta_j|)^3}\sum\limits_{q=2}^m\sum\limits_{i=1}^k\frac{1}{(1+|\omega-\eta^q_i|)^2}\\
&\ \ \ \ \ \ \ \ \ \ \ \ \ +\sum_{j=1}^k\frac{1}{(1+|\omega-\eta_j|)^2}\sum\limits_{q=2}^m\sum\limits_{i=1}^k\frac{1}{(1+|\omega-\eta^q_i|)^3}\bigg]\\
&\lesssim\frac{k}{\delta^3}\sum_{j=1}^k\sum_{q=2}^m\sum_{i=1}^k\frac{1}{|\eta_j-\eta_i^q|^\alpha}\bigg(\frac{1}{(1+|\omega-\eta_j|)^{5-\alpha}}+\frac{1}{(1+|\omega-\eta_i^q|)^{5-\alpha}}\bigg)\\
&\lesssim \frac{k}{\delta^3}\cdot\delta^\alpha \cdot k\sum_{q=1}^m\sum_{i=1}^k\frac{1}{(1+|\omega-\eta_i^q|)^{5-\alpha}}\\
&\lesssim\delta^\alpha k^2\sum\limits_{q=1}^m\sum\limits_{i=1}^k\frac{1}{(\delta+|x-\xi^q_i|)^3}\text{~for any~}1<\alpha\leq 2
\end{align*}
Last but not least, $$\begin{aligned}\bigg(\sum\limits_{q=1}^m\sum\limits_{i=1}^k\frac{1}{\delta+|x-\xi^q_i|}\bigg)^3\lesssim  k^2\sum\limits_{q=1}^m\sum\limits_{i=1}^k\frac{1}{(\delta+|x-\xi^q_i|)^3}.\end{aligned}$$
Collecting all the previous estimates the proof ends.
\end{proof}

\subsection{Solving the reduced problem}
We need to find  $(\delta,  \rho )$ such that ${\mathfrak c}_0={\mathfrak c}_1=0$ in the equations \eqref{exi}. We will follow the same strategy of 
 \cite{PWY}. In particular, we will find out a local Pohozaev's identity for the non-local equation \eqref{u11}. This is the   most technical part of the paper.
\\

First of all, arguing as in   \cite[Proposition 3.1]{PWY}, we prove that
\begin{proposition}\label{prop5}Let $D_{\varepsilon}:=\{x: |r-r_0|\leq\varepsilon, \varepsilon\in(2\sigma, 5\sigma)\}$ for some $\sigma>0$ (see \eqref{zet}).
Suppose that $(\delta,  \rho)$ satisfies
\begin{align}
&\int_{{\mathbb R}^4}(\mathcal L(\phi)-\mathcal E-\mathcal N(\phi))\frac{\partial W}{\partial \delta} dx=0,\label{po1}\\
&\int_{D_\varepsilon}(\mathcal L(\phi)-\mathcal E-\mathcal N(\phi))\langle x, \nabla u\rangle dx=0,\label{po2}
\end{align}
then ${\mathfrak c}_0={\mathfrak c}_1=0.$
 \end{proposition}

Next, we compute the first order term in the expansion of the L.H.S. of \eqref{po1}.
 \begin{proposition}\label{propred} There exist positive constants $\mathfrak a$ and $\mathfrak b$ such that
\begin{align}&\int\limits_{{\mathbb R}^4}(\mathcal L(\phi)-\mathcal E-\mathcal N(\phi)) \frac{\partial W}{\partial \delta}=k(-\mathfrak a\delta\ln \delta V(\rho)-\mathfrak b\delta\frac{k^2}{\rho^2})+o(\delta k^3).\label{ex1}
\end{align}
\end{proposition}
\begin{proof}
We have 
\begin{equation}\begin{aligned}& \int\limits_{{\mathbb R}^4}(\mathcal L(\phi)-\mathcal E-\mathcal N(\phi)) \frac{\partial W}{\partial \delta}\\&=\underbrace{\int\limits_{{\mathbb R}^4}V(|x|)W \frac{\partial W}{\partial \delta}}_{:=\mathcal M_1}-\underbrace{\int\limits_{{\mathbb R}^4}(W^3+\Delta W) \frac{\partial W}{\partial \delta}}_{:=\mathcal M_2}\\ &+\underbrace{\int_{{\mathbb R}^4}\(-\Delta\phi-3W^2\phi+V(x)\phi-\beta (W+\phi)\sum_{q=2}^m(W({\mathscr S}_qx)+\phi({\mathscr S}_qx))^2 -\phi^3-3W\phi^2\)\frac{\partial W}{\partial \delta} dx}_{:=\mathcal M_3}.
\label{cla}\end{aligned}\end{equation}

First we prove that
 
\begin{equation}\label{m3} \mathcal M_3
=o(\delta k^3).
\end{equation}
By \eqref{dis} and Lemma \ref{app2} and the fact that $\Big|\frac{\partial W}{\partial \delta}\Big|\lesssim \sum\limits_{p=1}^k \frac{1}{\delta^2+|x-\xi_p|^2}$
\begin{align}
&\int_{{\mathbb R}^4}W\sum_{q=2}^mW^2({\mathscr S}_qx)\Big| \frac{\partial W}{\partial \delta}\Big|dx\nonumber\\
&\lesssim k^2\int_{ {\mathbb R}^4}\sum_{j=1}^k\frac{\delta}{(\delta^2+|x-\xi_j|^2)^2}\sum_{q=2}^m\sum_{i=1}^k\frac{\delta^2}{(\delta^2+|x-\xi^q_i|^2)^2}\nonumber\\
&\lesssim \delta^{-1}k^2\int_{\mathbb R^4}\sum_{j=1}^k\sum_{q=2}^m\sum_{i=1}^k\frac{ 1}{|\eta_j-\eta_i^q|^\alpha} \bigg(\frac{1}{(1+|\omega-\eta_j|)^{8-\alpha}}+\frac{1}{(1+|\omega-\eta_i^q|)^{8-\alpha}}\bigg)  d\omega \nonumber\\
&\lesssim \delta^{\alpha-1}k^{4}\text{~for any $\alpha\in(1, 4)$,}\nonumber\\
&=o\( \delta k^3\)\ \text{~if we choose $\alpha\in(2, 4)$,}\nonumber\end{align}

and from $\|\phi\|_*=\|\phi({\mathscr S}_qx)\|_*$, Proposition \ref{propfix}
\begin{align*}
&\int_{{\mathbb R}^4}W(\phi^2+\sum_{q=2}^m\phi^2({\mathscr S}_qx))\Big| \frac{\partial W}{\partial \delta}\Big|dx\nonumber\\
&\lesssim \|\phi\|_*^2\int_{ \mathbb R^4}\sum_{i=1}^k\frac{\delta}{\delta^2+|x-\xi_i|^2}\sum_{q=1}^m\bigg(\sum_{p=1}^m\sum_{j=1}^k\frac{1}{\delta+|x-{\mathscr S}_q^{-1}\xi^p_j|}\bigg)^2 \sum_{p=1}^k \frac{1}{\delta^2+|x-\xi_p|^2}\nonumber\\
&\lesssim k\|\phi\|_*^2\int_{\mathbb R^4} \frac{\delta}{(\delta+|x-\xi_1|)^6} \nonumber\\
&= \mathcal O(\delta k),
\end{align*}
and \begin{align*}
&\int_{{\mathbb R}^4}\(\phi^3+\phi\sum_{q=2}^m\phi^2({\mathscr S}_qx)\)\Big| \frac{\partial W}{\partial \delta}\Big|dx\lesssim k^4\|\phi\|_*^3\int_{ {\mathbb R}^4} \frac{1}{(\delta+|x-\xi_1|)^5}= \mathcal O(\delta^2 k^4).
\end{align*}
Moreover  since $V, W, \phi$ satisfy \eqref{sy55}, by \eqref{vpz}-\eqref{deph} and Proposition \ref{propfix},
\begin{align*}
&\int_{{\mathbb R}^4}(-\Delta\phi-3W^2\phi+V(x)\phi+\phi\sum_{q=2}^mW^2({\mathscr S}_qx))\frac{\partial W}{\partial \delta}dx\\
&=k\int_{{\mathbb R}^4}(-\Delta\phi-3W^2\phi+V(x)\phi +\phi\sum_{q=2}^mW^2({\mathscr S}_qx))Z^0_{\delta, \xi_1}dx\\
&=\mathcal O(\delta k^{2+\tau})+o(\delta)=\mathcal O(\delta k^{2+\tau})\text{~for some small~}\tau>0.
\end{align*}
By all the above estimates \eqref{m3} follows.\\

Next,  let us estimate $\mathcal M_1$. \\
\begin{align*}
\mathcal M_1
=&k\bigg[\int\limits_{\mathbb R^4}V(x)\chi^2U_{\delta, \xi_1}\frac{\partial U_{\delta, \xi_1}}{\partial \delta} dx+\int\limits_{\mathbb R^4}V(x)\chi^2U_{\delta, \xi_1}\sum_{i=2}^k\frac{\partial U_{\delta, \xi_i}}{\partial \delta} dx\bigg]\nonumber\\
\sim&kV(\rho)\int_{|r-r_0|\leq\sigma}U_{\delta, \xi_1}\frac{\partial U_{\delta, \xi_1}}{\partial \delta} dx+\mathcal O(\delta k^2\ln k)
\nonumber\\
\sim&kV(\rho)\delta\int_{B(0, {\frac{\sigma}{\delta }})}-U(\langle y, \nabla U\rangle +U)dx+\mathcal O(\delta k^2\ln k)\nonumber\\
\sim&kV(\rho)\delta\int_{B(0, {\frac{\sigma}{\delta }})} U^2 dx\nonumber\\
\sim&-\mathfrak akV(\rho)\delta\ln\delta 
\end{align*}
where $\mathfrak a:={\mathtt c}^2 >0$ (see \eqref{udx}). 
Indeed, 
taking $x=|\xi_1-\xi_j|y+\xi_1$, then there exists $R>0$ such that $|y|\leq  Rk$ if $x\in \{x: ||x|-r_0|\leq \sigma\}$ since $|\xi_1-\xi_i|\geq \frac{\zeta}{k}$ for $i\neq 1$, and  then
$$\begin{aligned}
\int\limits_{\mathbb R^4}V(x)\chi^2U_{\delta, \xi_1}\sum_{i=2}^k\frac{\partial U_{\delta, \xi_i}}{\partial \delta} dx\lesssim &\delta k\int\limits_{\{x: ||x|-r_0|\leq \sigma\}}\frac{1}{|x-\xi_1|^2}\frac{1}{|x-\xi_i|^2}dx\\
\lesssim &\delta k\int\limits_{B(0, Rk)}\frac{1}{|y|^2}\frac{1}{|y+\frac{\xi_1-\xi_i}{|\xi_1-\xi_i|}|^2}dy\\
\lesssim &\delta k\bigg[ \int\limits_{|y|\leq 2}\frac{1}{|y|^2}\frac{1}{|y+\frac{\xi_1-\xi_i}{|\xi_1-\xi_i|}|^2}dy+\int\limits_{2\leq |y|\leq  Rk}\frac{1}{|y|^4}dy\bigg]\\
=&\delta k(\mathcal O(1)+\mathcal O(\ln k))=\mathcal O(\delta k\ln k).
\end{aligned} $$

Let us estimate $\mathcal M_2$. by Taylor's expansion   for fixed small $\zeta>0$ and for any $y\in B(0, \frac{\zeta}{\delta k}), $ \begin{align*}
&\sum_{i=2}^kU_{\delta, \xi_i}(\delta y+\xi_1)\\
&= {\mathtt c}\sum_{i=2}^k\frac{\delta}{|\xi_1-\xi_i|^2}\bigg[1-\frac{\delta^2+\delta^2|y|^2-2\delta\langle y, \xi_1-\xi_i\rangle}{|\xi_1-\xi_i|^2}\\
&\ \ \ +\mathcal O\bigg(\bigg(\frac{\delta^2+\delta^2|y|^2-2\delta\langle y, \xi_1-\xi_i\rangle}{|\xi_1-\xi_i|^2}\bigg)^2\bigg)\bigg],
\end{align*}
so \begin{align}
\mathcal M_2\sim&\int_{supp\{\chi\}}\(\(\sum_{i=1}^kU_{\delta, \xi_i}\)^3-\sum_{i=1}^kU_{\delta, \xi_i}^3\)\sum_{j=1}^k\frac{\partial U_{\delta, \xi_i}}{\partial \delta}dx\nonumber\\
\sim&k\int_{supp\{\chi\}}\(\(\sum_{i=1}^kU_{\delta, \xi_i}\)^3-\sum_{i=1}^kU_{\delta, \xi_i}^3\)\frac{\partial U_{\delta, \xi_1}}{\partial \delta}dx\nonumber\\
\sim&3k\int_{supp\{\chi\}\cap\mathbb R^4}\sum_{i=2}^k U_{\delta, \xi_1}^2U_{\delta, \xi_i} \frac{\partial U_{\delta, \xi_1}}{\partial \delta}+\sum_{i=2}^k U_{\delta, \xi_1}U^2_{\delta, \xi_i} \frac{\partial U_{\delta, \xi_1}}{\partial \delta}+\sum_{i\neq j, i, j\neq 1}U_{\delta, \xi_i}^2 U_{\delta, \xi_j}\frac{\partial U_{\delta, \xi_1}}{\partial \delta}dx\nonumber\\
\sim&3{\mathtt c}\delta k\sum_{i=2}^k\frac{1}{|\xi_1-\xi_i|^2}\int_{B(0, \frac{\zeta}{\delta k})} -U^2(\langle y, \nabla U\rangle +U)dy+\mathcal O(\delta^3 k^7)\nonumber\\
\sim&3{\mathtt c}\delta k\sum_{i=2}^k\frac{1}{|\xi_1-\xi_i|^2}\int_{B(0, \frac{\zeta}{\delta k})}\frac{1}{3}U^3 dy\nonumber\\
\sim&\mathfrak b\delta \frac{k^3}{\rho^2}\label{ma}
\end{align}
where $\mathfrak b:={\mathtt c}C_2\int_{\mathbb R^4}U^3dx >0$ (see \eqref{tau}). Indeed,  we have
\begin{align*}
&\sum_{i=2}^k\bigg[\int_{B(\xi_i, \frac{\zeta}{k})} U_{\delta, \xi_1}^2U_{\delta, \xi_i} \frac{\partial U_{\delta, \xi_1}}{\partial \delta}+\int_{(B(\xi_1, \frac{\zeta}{k})\cup B(\xi_i, \frac{\zeta}{k}))^c} U_{\delta, \xi_1}^2U_{\delta, \xi_i} \frac{\partial U_{\delta, \xi_1}}{\partial \delta}\bigg]\\
&\lesssim \delta^5k^7\int_{B(0, \frac{\zeta}{\delta k})} \frac{1}{1+|y|^2}dy+\delta k^3\int_{B^c(0, \frac{\zeta}{\delta k})} \frac{1}{1+|y|^6}dy\\
&\lesssim \delta^3k^5
\end{align*}
and using Lemma \ref{app1}, it holds
\begin{align*}
&\int_{supp\{\chi\}\cap\mathbb R^4}\sum_{i=2}^k U_{\delta, \xi_1}U^2_{\delta, \xi_i} \frac{\partial U_{\delta, \xi_1}}{\partial \delta}dx\\
&\lesssim \frac{1}{\delta}\sum_{i=2}^k\int_{|r-r_0|\leq \frac{\sigma}{\delta}}\frac{1}{(1+|\omega-\eta_1|)^4}\frac{1}{(1+|\omega-\eta_i|)^4}d\omega\\
&\lesssim\frac{1}{\delta}\sum_{i=2}^k\int_{|r-r_0|\leq \frac{\sigma}{\delta}}\frac{1}{|\eta_1-\eta_i|^4}\bigg(\frac{1}{(1+|\omega-\eta_1|)^4}+\frac{1}{(1+|\omega-\eta_i|)^4}\bigg)d\omega\\
&\lesssim \delta^3k^4|\ln \delta|,
\end{align*}
 \begin{align*}
&\int_{|r-r_0|\leq {\sigma}}\sum_{i\neq j, i, j\neq 1}U_{\delta, \xi_i}^2 U_{\delta, \xi_j}\frac{\partial U_{\delta, \xi_1}}{\partial \delta}dx\\
&\lesssim \frac{1}{\delta}\sum_{i\neq j, i, j\neq 1}\int_{|r-r_0|\leq \frac{\sigma}{\delta}}\frac{1}{|\eta_i-\eta_j|^2}\bigg(\frac{1}{(1+|\omega-\eta_i|)^4}+\frac{1}{(1+|\omega-\eta_j|)^4}\bigg)\frac{1}{(1+|\omega-\eta_1|)^2}d\omega\\
&\lesssim \delta^3k^4\int_{|r-r_0|\leq \frac{\sigma}{\delta}}\frac{1}{(1+|\omega-\eta_i|)^4}+\frac{1}{(1+|\omega-\eta_j|)^4}+\frac{1}{(1+|\omega-\eta_1|)^4}d\omega\\
&\lesssim \delta^3k^4|\ln \delta|.
\end{align*}
Finally, \eqref{ex1} follows  by \eqref{cla}-\eqref{ma}.

 \end{proof}

Now, let us compute the first order term in  the  expansion of the L.H.S. of \eqref{po2}.

 \begin{proposition}\label{propexp}There exists a positive constant $ \mathfrak c$ such that
 \begin{align}
&\int_{D_\varepsilon}(\mathcal L(\phi)-\mathcal E-\mathcal N(\phi))\langle x, \nabla u\rangle dx=\frac{1}{2 \rho  }\frac{\partial ( \rho  ^2V( \rho))}{\partial  \rho  }\mathfrak c k\delta^2\ln \delta+o(\delta^2k^3).\label{pooo2}
 \end{align}

 \end{proposition}
 \begin{proof}
Denote $u_q(x):=u({\mathscr S}_qx), q=1,...m$ where $u_1(x)=u(x)$. Recalling that \begin{align}\label{ore}\mathcal L(\phi)-\mathcal E-\mathcal N(\phi)=-\Delta u+V(x)u-u^3-\beta u \sum_{q=2}^m u^2({\mathscr S}_qx).\end{align}
For the sake of clarity, we divide the proof into several steps.
\vspace{1mm}

\noindent \textbf{Step 1:} Let us prove that for any $q=2,...,m$,
\begin{align}
&\int_{D_\varepsilon}\beta u \sum_{q=2}^m u^2({\mathscr S}_qx)\langle x, \nabla u\rangle dx=\int_{D_\varepsilon}\beta u_q\sum_{p\neq q} u_p^2\langle x, \nabla u_q\rangle dx,\ \label{poo2}
\end{align}
First of all we point out that for any 
 $q=2,...,m$, \begin{align*}u_q(x)=&u({\mathscr S}_qx)\\
=&u(\cos\frac{\pi(q-1)}{mk}x_1-\sin\frac{\pi(q-1)}{mk}x_2, \sin\frac{\pi(q-1)}{mk}x_1+\cos\frac{\pi(q-1)}{mk}x_2,\\
&\cos\frac{\pi(q-1)}{mk}x_3+\sin\frac{\pi(q-1)}{mk}x_4, -\sin\frac{\pi(q-1)}{mk}x_3+\cos\frac{\pi(q-1)}{mk}x_4)\\
:=&u(y_1, y_2, y_3, y_4) \text{~where~} y={\mathscr S}_qx
\end{align*}
and
\begin{align*}u({\mathscr S}_{m+1}x)=u({\mathscr R}_{\frac{k}{2}+1}x)=u(x), V(x)=V(|x|)=V({\mathscr S}_qx), q=2,...,m.\end{align*}
Next, since
\begin{align*}&
\nabla_x u_q=\nabla_yu\cdot {\mathscr S}_q \\
&\frac{\partial^2 u_q}{\partial x^2_1}=\cos^2\frac{\pi(q-1)}{mk}\frac{\partial^2 u}{\partial y^2_1}+\cos\frac{\pi(q-1)}{mk}\sin\frac{\pi(q-1)}{mk}\frac{\partial^2 u}{\partial y_1\partial y_2}\\
&\ \ \ \ \ \ \ \ \ \  +\sin\frac{\pi(q-1)}{mk}\cos\frac{\pi(q-1)}{mk}\frac{\partial^2 u}{\partial y_2\partial y_1}+\sin^2\frac{\pi(q-1)}{mk}\frac{\partial^2 u_1}{\partial y^2_2}\\
&\frac{\partial^2 u_q}{\partial x^2_2}=\sin^2\frac{\pi(q-1)}{mk}\frac{\partial^2 u}{\partial y^2_1}-\cos\frac{\pi(q-1)}{mk}\sin\frac{\pi(q-1)}{mk}\frac{\partial^2 u}{\partial y_1\partial y_2}\\
&\ \ \ \ \ \ \ \ \ \  -\sin\frac{\pi(q-1)}{mk}\cos\frac{\pi(q-1)}{mk}\frac{\partial^2 u}{\partial y_2\partial y_1}+\cos^2\frac{\pi(q-1)}{mk}\frac{\partial^2 u}{\partial y^2_2},\\
&\frac{\partial^2 u_q}{\partial x^2_3}=\cos^2\frac{\pi(q-1)}{mk}\frac{\partial^2 u}{\partial y^2_3}-\cos\frac{\pi(q-1)}{mk}\sin\frac{\pi(q-1)}{mk}\frac{\partial^2 u}{\partial y_3\partial y_4}\\
&\ \ \ \ \ \ \ \ \ \  -\sin\frac{\pi(q-1)}{mk}\cos\frac{\pi(q-1)}{mk}\frac{\partial^2 u}{\partial y_4\partial y_3}+\sin^2\frac{\pi(q-1)}{mk}\frac{\partial^2 u_1}{\partial y^2_4}\\
&\frac{\partial^2 u_q}{\partial x^2_4}=\sin^2\frac{\pi(q-1)}{mk}\frac{\partial^2 u}{\partial y^2_3}+\cos\frac{\pi(q-1)}{mk}\sin\frac{\pi(q-1)}{mk}\frac{\partial^2 u}{\partial y_3\partial y_4}\\
&\ \ \ \ \ \ \ \ \ \  +\sin\frac{\pi(q-1)}{mk}\cos\frac{\pi(q-1)}{mk}\frac{\partial^2 u}{\partial y_4\partial y_3}+\cos^2\frac{\pi(q-1)}{mk}\frac{\partial^2 u}{\partial y^2_4},
\end{align*}
we have \begin{align*}-\Delta u(y)=-\Delta u_q(x),\end{align*}
and
\begin{align*}
\langle x, \nabla_x u_q\rangle=\langle x, \nabla_yu\cdot {\mathscr S}_q\rangle=\langle {\mathscr S}_q^{-1}y,  \nabla_yu\cdot {\mathscr S}_q\rangle=\langle y, \nabla_yu\rangle
\end{align*}
Finally
\begin{align*}
\int_{D_\varepsilon}\beta u_q\sum_{p\neq q} u_p^2\langle x, \nabla u_q\rangle dx&\overset{y={\mathscr S}_qx}{=}\int_{D_\varepsilon}\beta u \sum_{q=2}^m u^2({\mathscr S}_qy)\langle y, \nabla u\rangle dy\\&\ \ =\int_{D_\varepsilon}\beta u \sum_{q=2}^m u^2({\mathscr S}_qx)\langle x, \nabla u\rangle dx \end{align*}
and the claim is proved.
 \vspace{1mm}

\noindent \textbf{Step 2:} Let us prove that \begin{align}&\int_{D_\varepsilon}(-\Delta u+V(x)u-u^3-\beta u \sum_{q=2}^m u^2({\mathscr S}_qx))\langle x, \nabla u\rangle dx\nonumber\\&=\int_{D_{\varepsilon}}-|\nabla u|^2- V(x)u^2-(V(x) + \frac12\langle \nabla V,  x\rangle)u^2+u^4+\beta u^2\sum_{q=2}^mu^2({\mathscr S}_qx)dx\nonumber\\&\ \ \ \ +\mathcal O\(\int_{\partial D_\varepsilon}|\nabla \phi|^2+|\phi|^2+|\phi|^4\).\label{equ1}
\end{align}
Indeed,  for any $q=1,...,m$
\begin{align*}
&\int_{D_\varepsilon}\beta u_q \sum_{p\neq q} u^2_p\langle x, \nabla u_q\rangle dx\nonumber\\&=\int_{D_\varepsilon}\beta u_q \sum_{p\neq q} u^2_p\sum_{l=1}^4 x_l \frac{\partial u_q}{\partial x_l}dx =\int_{D_\varepsilon}\beta  \sum_{p\neq q} u^2_p\sum_{l=1}^4 x_l \frac{\partial(\frac12 u_q^2)}{\partial x_l}dx\nonumber\\
&=\int_{\partial D_\varepsilon}\frac12\beta  u_q^2 \sum_{p\neq q} u^2_p\langle x, \nu\rangle dS-\int_{D_\varepsilon}\beta \frac12 u_q^2\sum_{l=1}^4 \left[ \sum_{p\neq q} u^2_p+x_l 2 \sum_{p\neq q} u_p \frac{\partial u_p}{\partial x_l}\right]dx\nonumber\\
&=\int_{\partial D_\varepsilon}\frac12\beta \phi_q^2 \sum_{p\neq q} \phi^2_p \langle x, \nu\rangle dS-\int_{D_\varepsilon}2\beta  u_q^2 \sum_{p\neq q} u^2_p+\beta u_q^2 \sum_{p\neq q} u_p\langle x, \nabla u_p\rangle dx\nonumber\\
&=\int_{\partial D_\varepsilon}\frac12\beta \sum_{q=1}^m\phi_q^2 \sum_{p\neq q} \phi^2_p \langle x, \nu\rangle dS-\int_{D_\varepsilon}2\beta \sum_{q=1}^m u_q^2 \sum_{p\neq q} u^2_p+\beta\sum_{q=1}^m u_q^2 \sum_{p\neq q} u_p\langle x, \nabla u_p\rangle dx\\
& =\int_{\partial D_\varepsilon}\frac14\beta \sum_{q=1}^m\phi_q^2 \sum_{p\neq q} \phi^2_p \langle x, \nu\rangle dS-\int_{D_\varepsilon}\beta \sum_{q=1}^m u_q^2 \sum_{p\neq q} u^2_p dx
\end{align*}
since $$\int_{D_\varepsilon}\beta \sum_{q=1}^mu_q\sum_{p\neq q} u_p^2\langle x, \nabla u_q\rangle dx=\int_{D_\varepsilon}\beta\sum_{q=1}^m u_q^2 \sum_{p\neq q} u_p\langle x, \nabla u_p\rangle dx.$$
Then by \eqref{poo2}
\begin{align}
&\int_{D_\varepsilon}\beta u \sum_{q=2}^m u^2({\mathscr S}_qx)\langle x, \nabla u_q\rangle dx =\frac{1}{m}\int_{D_\varepsilon}\beta \sum_{q=1}^mu_q\sum_{p\neq q} u_p^2\langle x, \nabla u_q\rangle dx\nonumber\\
&=\frac{1}{m}\bigg[\int_{\partial D_\varepsilon}\frac14\beta \sum_{q=1}^m\phi_q^2 \sum_{p\neq q} \phi^2_p \langle x, \nu\rangle dS-\int_{D_\varepsilon}\beta \sum_{q=1}^m u_q^2 \sum_{p\neq q} u^2_p dx\bigg]\nonumber\\
&=\frac14\int_{\partial D_\varepsilon}\beta \phi^2 \sum_{q=2} \phi^2({\mathscr S}_qx) \langle x, \nu\rangle dS-\int_{D_\varepsilon}\beta   u^2 \sum_{q=2}^m u^2_q dx.\label{wan3}
\end{align}

By standard computations taking into account that $u=\phi$ on $\partial D_\varepsilon$ (due to the presence of the cut-off function \eqref{zet}),
\begin{align}&\int_{D_\varepsilon}(-\Delta u+V(x)u-u^3)\langle x, \nabla_x u\rangle dx\nonumber\\&=\int_{D_{\varepsilon}}\(-|\nabla u|^2-\frac12(4V(x) + \langle \nabla V,  x\rangle)u^2+u^4\)dx\nonumber\\
&\ \ \ +\mathcal O\(\int_{\partial D_\varepsilon}\(|\nabla \phi|^2+|\phi|^2+|\phi|^4\) dS\)\label{wan2}\end{align}
Finally,  \eqref{equ1} follows by \eqref{wan3} and \eqref{wan2}.

\vspace{1mm}

\noindent \textbf{Step 3:} Let us prove that
 \begin{equation}\begin{aligned}\label{y1y1}
\int_{D_\varepsilon}\(|\nabla u|^2+V(x)u^2\)dx=&\int_{D_\varepsilon}\(u^4+\beta u^2 \sum_{q=2}^m u^2({\mathscr S}_qx)\)dx\\ &+\mathcal O\(\int_{\partial D_\varepsilon}\(|\nabla \phi|^2+|\phi|^2\)dx\)+o(\delta^2 k^3).
\end{aligned}\end{equation}
First of all, testing \eqref{exi} by $u$ and using \eqref{ore}, we get
 \begin{align*}
\int_{D_\varepsilon}\(|\nabla u|^2+V(x)u^2\)dx=&\int_{D_\varepsilon}\(u^4+\beta u^2 \sum_{q=2}^m u^2({\mathscr S}_qx)\)dx+\sum_{l=0}^1\sum_{i=1}^k{\mathfrak c}_l(\chi U_{\delta,\xi_i})^2Z^l_{\delta, \xi_i} udx\\
&+\mathcal O\(\int_{\partial D_\varepsilon}\(|\nabla \phi|^2+|\phi|^2\)dx\).
\end{align*}

Now we claim
\begin{align*}
\int_{D_\varepsilon}\sum_{i=1}^k{\mathfrak c}_l(\chi U_{\delta,\xi_i})^2Z^l_{\delta, \xi_i}W=o(\delta^2 k^3), l=0, 1.\end{align*}
Indeed, it holds
\begin{align*}
\int_{D_\varepsilon}\sum_{i=1}^k(\chi U_{\delta,\xi_i})^2Z^0_{\delta, \xi_i}Wdx\sim&k\int_{D_\varepsilon}(\chi U_{\delta,\xi_i})^3Z^0_{\delta, \xi_i}dx\sim k\int_{D_\varepsilon}\frac14\frac{\partial (\chi U_{\delta,\xi_i})^4}{\partial \delta}dx\\
\sim&k\bigg[\int_{D_\varepsilon}\frac14\frac{\partial  U_{\delta,\xi_i}^4}{\partial \delta}dx+\int_{D_\varepsilon}\frac14(\chi^4-1)\frac{\partial U_{\delta,\xi_i}^4}{\partial \delta}dx\bigg]\\
\sim&\frac k4\bigg[\frac{\partial }{\partial \delta}\int_{D_\varepsilon}  U_{\delta,\xi_i}^4dx+\int_{D_\varepsilon}(\chi^4-1)\frac{\partial U_{\delta,\xi_i}^4}{\partial \delta}dx\bigg]\\
 =&\mathcal O(\delta^3k)
\end{align*}
and 
\begin{align*}
\int_{D_\varepsilon}\sum_{i=1}^k(\chi U_{\delta,\xi_i})^2Z^1_{\delta, \xi_i}Wdx&\sim k\int_{D_\varepsilon}\frac14\frac{\partial (\chi U_{\delta,\xi_i})^4}{\partial \rho}dx\\
&\sim k\bigg[\int_{D_\varepsilon}\frac14\frac{\partial  U_{\delta,\xi_i}^4}{\partial \rho}dx+\int_{D_\varepsilon}\frac14(\chi^4-1)\frac{\partial U_{\delta,\xi_i}^4}{\partial \rho}dx\bigg]\\
&\sim k\int_{D_\varepsilon}\frac14\frac{\partial  U_{\delta,\xi_i}^4}{\partial \rho}dx+\mathcal O(\delta^3k)\\
&=\mathcal O(\delta^3k)
\end{align*}
The claim follows since  ${\mathfrak c}_l=\mathcal O(\delta^2), l=0,1$ (see Proposition \ref{propfix}) and the orthogonality condition on $\phi.$ 

\noindent \textbf{Step 4:} By \eqref{equ1}  and \eqref{y1y1}
\begin{align*}&\int_{D_\varepsilon}(-\Delta u+V(x)u-u^3-\beta u \sum_{q=2}^m u^2({\mathscr S}_qx))\langle x, \nabla u\rangle dx\nonumber \\&=\int_{D_\varepsilon}-(V(x) + \frac12\langle \nabla V(x),  x\rangle)u^2dx+\mathcal O\(\int_{\partial D_\varepsilon}\(|\nabla \phi|^2+|\phi|^2+|\phi|^4\)dS\)+o(\delta^2k^3)\nonumber\\
&=\int_{D_\varepsilon}-\frac{1}{2r}\frac{\partial (r^2V(r))}{\partial r}u^2dx+\mathcal O\(\int_{\partial D_\varepsilon}\(|\nabla \phi|^2+|\phi|^2+|\phi|^4\)dS\)+o(\delta^2k^3).
\end{align*}
\noindent \textbf{Step 5:} Let us prove that \begin{align*}\int_{\partial D_\varepsilon}\(|\nabla \phi|^2+\phi^2+\phi^4 \)dS=\mathcal O(\delta^2k^2). \end{align*}
Indeed, for any bounded domain $D\subset \mathbb R^4$,
\begin{equation*}
\begin{aligned}
\int_{D} \phi^4dx&\leq  \|\phi\|^{4}_*\int_{D}\(\sum_{q=1}^m\sum_{i=1}^k\frac{1}{\delta+|x-\xi^q_i|}\)^4dx \\
&\lesssim \delta^{4}\bigg[\int_{D }\sum_{q=1}^m\sum_{i=1}^k\(\frac{1}{\delta+|x-\xi^q_i|}\)^{4}dx\\
&\ \ \ \ \ +\int_{D}\sum_{(p, j)\neq (q, i)}\(\frac{1}{\delta+|x-\xi^q_i|}\)^{\alpha_1}\(\frac{1}{\delta+|x-\xi^p_j|}\)^{\alpha_2}dx\bigg] \\
&\lesssim \delta^{4} \bigg[k\int_{B(0, \frac{R}{\delta})}\frac{1}{(1+|y|)^{4}}dy+\int_{D}\frac{1}{\delta^4}\sum_{(p, j)\neq (q, i)}\frac{\delta^\alpha}{|\xi^p_j-\xi^q_i|^{\alpha}}\\
&\ \ \ \ \ \cdot \(\frac{1}{1+|\frac{x-\xi^p_j}{\delta}|^{4-\alpha}}+\frac{1}{1+|\frac{x-\xi^q_i}{\delta}|^{4-\alpha}} \)dx\bigg] \\
&=\mathcal O(\delta^4|\ln\delta|k) \text{~where we choose~}0<\alpha \ll1\text{~for any~}(\alpha_1, \alpha_2)\text{~with~}\alpha_1+\alpha_2=4,
\end{aligned}
\end{equation*}
and \begin{align}
\int_{D} \phi^{2} dx&\leq \|\phi\|^{2}_*\int_{D}\bigg(\sum_{q=1}^m\sum_{i=1}^k\frac{1}{\delta+|x-\xi^q_i|}\bigg)^{2}dx\nonumber \\
&\lesssim k^{2}\delta^{2}\int_{D}\(\frac{1}{\delta+|x-\xi^q_i|}\)^{2}dx\nonumber\\
&=\mathcal O(\delta^{2}k^{2}).\label{phhi}\end{align}
%
Now, we multiply \eqref{exi} by $\phi$ and integrate over $D_{4\sigma}\setminus D_{3\sigma}$ and taking into account that $W=0$ in $D_{5\sigma}\setminus D_{2\sigma}$ (because of the choice of the cut-off function \eqref{zet}) we get
 $$\int_{D_{4\sigma}\setminus D_{3\sigma}} |\nabla \phi|^2dx \lesssim \int_{D_{5\sigma}\setminus D_{2\sigma}}  \(\phi^2+\phi^4\)dx=\mathcal O(\delta^2k^2).$$
 So there exists $\varepsilon\in (3\sigma, 4\sigma)$ such that $$\int_{\partial D_\varepsilon} |\nabla\phi|^2dS=\mathcal O(\delta^2k^2)\ \hbox{and}\ \int_{\partial D_\varepsilon} \phi^2+\phi^4dS=\mathcal O(\delta^2k^2) $$
 and the claim follows.\\
 
 \noindent \textbf{Step 6:} There is a constant $\mathfrak c>0$ such that
\begin{align}\label{exxp} \int_{D_\varepsilon}-\frac{1}{2r}\frac{\partial (r^2V(r))}{\partial r}u^2 dx=k(\frac{1}{2\rho}\frac{\partial (\rho^2V(\rho))}{\partial \rho}\mathfrak c\delta^2\ln\delta+o(\delta^2|\ln\delta|)).\end{align}
 Indeed, by  Lemma \ref{app1} and \eqref{tau} 
 \begin{align*}
&\int_{D_\varepsilon}\sum\limits_{i\neq j}U_{\delta, \xi_i}\cdot U_{\delta, \xi_j} dx\\
&=k\int_{D_\varepsilon}\sum\limits_{j\neq 1} U_{\delta, \xi_1}\cdot  U_{\delta, \xi_j} dx\\
&\lesssim  k\delta^2\int_{\frac{D_\varepsilon }{\delta}}\sum\limits_{j\neq 1}\frac{1}{(1+|\omega-\eta_1|)^2}\frac{1}{(1+|\omega-\eta_j|)^2} d\omega\\
&\lesssim k\delta^2\int_{\frac{D_\varepsilon }{\delta}}\sum\limits_{j\neq 1}\frac{1}{(|\eta_1-\eta_j|)^\tau}\(\frac{1}{(1+|\omega-\eta_1|)^{4-\tau}}+\frac{1}{(1+|\omega-\eta_j|)^{4-\tau}}\) d\omega \\
&\lesssim \delta^2k^{2+\tau}\text{~for small $\tau>0$}
 \end{align*}
 and then
 \begin{align*}
&\int_{D_\varepsilon}-\frac{1}{2r}\frac{\partial (r^2V(r))}{\partial r}W^2 dx\\&=\int_{D_\varepsilon}-\frac{1}{2r}\frac{\partial (r^2V(r))}{\partial r}\sum\limits_{i=1}^k(\chi U_{\delta, \xi_i})^2 dx+\int_{D_\varepsilon}-\frac{1}{2r}\frac{\partial (r^2V(r))}{\partial r}\sum\limits_{i\neq j}\chi U_{\delta, \xi_i}\cdot \chi U_{\delta, \xi_j} dx\\
&=-\frac{1}{2\rho}\frac{\partial (\rho^2V(\rho))}{\partial \rho}k\int_{D_\varepsilon}\chi^2U^2_{\delta,\xi_i} dx\\
&\ \ \ \ \ +k\int_{D_\varepsilon}\(\frac{1}{2\rho}\frac{\partial (\rho^2V(\rho)}{\partial \rho}-\frac{1}{2r}\frac{\partial (r^2V(r))}{\partial r}\)(\chi U_{\delta, \xi_i})^2 dx+{\mathcal O(\delta^2k^{2+\tau})}\\
&=-\frac{1}{2\rho}\frac{\partial (\rho^2V(\rho))}{\partial \rho}k\left[\int_{D_\varepsilon}U^2_{\delta,\xi_i} dx+\int_{D_\varepsilon}(\chi^2-1)U^2_{\delta,\xi_i} dx\right]\\
&\ \ \ \ \ \ +\mathcal O\(k\int_{D_\varepsilon}||x|-\rho|U_{\delta, \xi_i}^2 dx\)+\mathcal O(\delta^2k^{2+\tau})\\
&=\frac{1}{2\rho}\frac{\partial (\rho^2V(\rho))}{\partial \rho} \mathfrak ck\delta^2\ln\delta +\mathcal O(\delta^2k)+\mathcal O(\delta^2k^{2+\tau}).
\end{align*}
where $\mathfrak c:={\mathtt c}^2>0$ (see \eqref{udx}).

 Next by Proposition \ref{propfix} and \eqref{phhi}  
 \begin{align*}
&\Big|\int_{D_\varepsilon}-\frac{1}{2r}\frac{\partial (r^2V(r))}{\partial r}\phi^2-\frac{1}{r}\frac{\partial (r^2V(r))}{\partial r}W\phi dx\Big|\\
& \lesssim \int_{D_\varepsilon} \phi^2 dx+\|\phi\|_*\int_{D_\varepsilon}\sum\limits_{i=1}^k\chi U_{\delta, \xi_i}\sum_{q=1}^m\sum_{j=1}^k\frac{1}{\delta+|x-\xi^q_j|} dx\\
& \lesssim \delta^2k^2.
 \end{align*} 
Finally \eqref{exxp} follows and the proof of  \eqref{pooo2} is complete.
 \end{proof}
 \subsection{Proof of Theorem \ref{thm1.1}: completed}\label{2.8}
 By \eqref{ex1}, \eqref{pooo2}  and Proposition \ref{prop5}, the problem reduces to finding $(\delta, \rho)$ such that 
$$\left\{ \begin{aligned}
 &-\mathfrak a\delta\ln\delta V(\rho)-\mathfrak b\delta\frac{ k^2}{\rho^2} =o(\delta k^3),\\
 &k\mathfrak c\delta^2\ln \delta \frac{1}{\rho}\frac{\partial(\rho ^2V(\rho))}{\partial \rho } =o(\delta^2k^3),
 \end{aligned}\right.$$
which is equivalent to finding $d=d_k>0$
(remember that  $\delta= e^{-d k^2}$ for some $d>0$)  and $\rho=\rho_k$ solutions of
 \begin{equation*}
\left\{ \begin{aligned}
&\mathfrak a dV(\rho)-\frac{\mathfrak b}{\rho^2}=o(1),\\
& \frac{\partial(\rho^2V(\rho))}{\partial \rho}=o(1).
\end{aligned}\right.\end{equation*}
Eventually, since $r_0$ is a non-degenerate critical point of the function $r^2V(r)$ with $V(r_0)>0,$
this last problem has a solution $d_k\sim \frac{\mathfrak b}{\mathfrak ar^2_0V(r_0 )}$ and 
 $\rho_k\sim r_0$ as $k$ is large enough. That concludes the proof.
 \qed

\end{document}